\documentclass[leqno,12pt]{article}
\usepackage{amsfonts}
\usepackage{amsmath, amssymb}
\usepackage{amsthm}

\theoremstyle{plain}
\newtheorem{theorem}{\indent\sc Theorem}[section]
\newtheorem{lemma}[theorem]{\indent\sc Lemma}
\newtheorem{corollary}[theorem]{\indent\sc Corollary}
\newtheorem{proposition}[theorem]{\indent\sc Proposition}

\theoremstyle{definition}

\newtheorem{remark}[theorem]{\indent\sc Remark}
\newtheorem{example}[theorem]{\indent\sc Example}

\newcommand\on{\operatorname}
\renewcommand\div{\on{div}}

\newcommand\Ric{\on{Ric}}
\newcommand\scal{\on{scal}}

\newcommand\Id{\on{Id}}

\newcommand\func{\operatorname}
\newcommand\grad{\func{grad}}

\begin{document}

\title{Remarks on Riemann and Ricci solitons\\
in $(\alpha,\beta)$-contact metric manifolds}
\author{Adara M. Blaga and Dan Radu La\c tcu}
\date{}
\maketitle

\begin{abstract}
We study almost Riemann solitons and almost Ricci solitons in
an $(\alpha,\beta)$-contact metric manifold satisfying some
Ricci symmetry conditions, treating the case when the potential
vector field of the soliton is pointwise collinear with
the structure vector field.
\end{abstract}

\markboth{{\small\it {\hspace{2cm} Remarks on Riemann and Ricci solitons
in $(\alpha,\beta)$-contact metric manifolds}}}{\small\it{Remarks on Riemann and Ricci solitons
in $(\alpha,\beta)$-contact metric manifolds
\hspace{2cm}}}


\section{Introduction}

Riemann and Ricci solitons are generalized fixed points of the Riemann and Ricci
flow, respectively. They are defined by a smooth vector field $V$ and a real
constant $\lambda$ which satisfy respectively the following equations
\begin{equation}\label{2}
\frac{1}{2}\pounds_{V}g\odot g+R=\lambda G
\qquad
\text{(Riemann soliton)}
\end{equation}
where $G:=\frac{1}{2}g\odot g$, $\pounds_{V}$ is the Lie derivative operator
in the direction of the vector field $V$, $R$ is the Riemann curvature of $g$,
and
\begin{equation}\label{1}
\frac{1}{2}\pounds_{V}g+\Ric=\lambda g
\qquad
\text{(Ricci soliton)}
\end{equation}
where $\Ric$ is the Ricci curvature of $g$. The above notation $\odot$ stands
for the Kulkarni-Nomizu product, which for two arbitrary $(0,2)$-tensor fields
$T_1$ and $T_2$ on $M$, is defined by
\pagebreak
$$
(T_1\odot T_2)(X,Y,Z,W):=T_1(X,W)T_2(Y,Z)+T_1(Y,Z)T_2(X,W)$$$$-T_1(X,Z)T_2(Y,W)-T_1(Y,W)T_2(X,Z),
$$ for any $X$, $Y$, $Z$, $W\in\mathfrak{X}(M)$, where $\mathfrak{X}(M)$ is the set of
all vector fields on $M$.

If $\lambda$ is a smooth function in~\eqref{2} and~\eqref{1},
then the soliton will be called almost Riemann and almost Ricci
soliton, respectively.


Some years ago Udri\c ste~\cite{ud} has studied the relations between Riemann flow and Ricci
flow. Then, in the context of contact geometry, Hiric\u a and Udri\c ste
proved~\cite{hi} that if a Sasakian manifold admits a Riemann soliton with
potential vector field pointwise collinear with the structure vector field
$\xi$, then it is a Sasakian space form. Recently, Devaraja,
Kumara and Venkatesha~\cite{de, v} and De and De~\cite{uday}, gave some geometric properties of Riemann solitons in
Kenmotsu, Sasakian and $K$-contact manifolds.


Inspired by the above studies, in the present paper, we establish some
properties of almost Riemann and almost Ricci solitons $(V,\lambda)$ in
an $(\alpha,\beta)$-contact metric manifold $(M,\phi,\xi,\eta,g)$ under
some additional assumptions: Ricci symmetry and $\phi$-Ricci symmetry.
We consider the case when the potential vector field $V$ of the soliton is
pointwise collinear with the structure vector field $\xi$, i.e., it belongs
to the vertical distribution $\langle\xi\rangle$ which is $g$-orthogonal to
the contact distribution $\ker(\eta)$. In particular, if $V$ is a constant
multiple of $\xi$, then any $\beta$-Sasakian manifold admitting an almost
Riemann or an almost Ricci soliton is an Einstein manifold.

\section{$(\alpha,\beta)$-contact metric structures}

Let $(M,\phi,\xi,\eta,g)$ be a $(2n+1)$-dimensional almost contact metric
manifold, i.e., a smooth manifold $M$ equipped with a $(1,1)$-tensor field
$\phi$, a vector field $\xi$, a $1$-form $\eta$ and a Riemannian metric $g$
satisfying~\cite{sas}
\[
\arraycolsep 2pt
\begin{array}{@{}c@{\qquad}rcl@{\qquad}rcl@{}}
\phi^2=-(\Id-\eta\otimes\xi),&
\eta(\xi)&=&1,&
g(\phi\cdot,\phi\cdot)&=&g-\eta\otimes\eta\\
\phi\xi=0,\qquad\eta\circ\phi=0,&
i_{\xi}g&=&\eta,&
g(\phi\cdot,\cdot)&=&-g(\cdot,\phi\cdot).
\end{array}
\]

If there exist two smooth functions $\alpha$ and $\beta$ on $M$ such that
the Levi-Civita connection $\nabla$ of $g$ satisfies
\begin{equation}\label{e2}
(\nabla_{X}\phi)Y=\alpha[g(\phi X,Y)\xi-\eta(Y)\phi X]+
\beta[g(X,Y)\xi-\eta(Y)X],
\end{equation}
for any $X$, $Y\in\mathfrak{X}(M)$, then we call $M$
\emph{an $(\alpha,\beta)$-contact metric manifold}
(called also \emph{trans-Sasakian manifold}~\cite{ou}).

In particular, if $\alpha=\beta=0$, $M$ is a \emph{cosymplectic} manifold;
if $\alpha=0$ and $\beta$ is a nonzero constant, $M$ is
a \emph{$\beta$-Sasakian} manifold (in particular, Sasakian if $\beta=1$),
if $\beta=0$ and $\alpha$ is a nonzero constant, $M$ is
an \emph{$\alpha$-Kenmotsu} manifold (in particular, Kenmotsu if $\alpha=1$).

By a direct computation, we obtain
\begin{equation}\label{16}
\nabla\xi=-\alpha\phi^2-\beta\phi,
\qquad
\pounds_{\xi}g=2\alpha(g-\eta\otimes\eta),
\qquad
\div(\xi)=2n\alpha.
\end{equation}

Remark that for any $X$, $Y\in\mathfrak{X}(M)$, we have
\[
(\nabla_X\phi)Y=g({F_{\alpha,\beta}X},Y)\xi-\eta(Y)F_{\alpha,\beta}X,
\qquad
\nabla_X\xi=-F_{\alpha,\beta}(\phi X),
\]
where $F_{\alpha,\beta}:=\alpha\phi+\beta\Id$. Also, notice that
\[
(\nabla_XF_{\alpha,\beta})Y=\alpha(\nabla_X\phi)Y+F_{X(\alpha),X(\beta)}Y.
\]

Let $(M,\phi,\xi,\eta,g)$ be a $(2n+1)$-dimensional
$(\alpha,\beta)$-contact metric manifold.

\begin{lemma}
If $V$ is a vector field on $M$ pointwise collinear with $\xi$, i.e.,
$V=\eta(V)\xi$, then
\begin{align}
\begin{split}
\nabla V&=[d (\eta(V))-\alpha \eta(V)\eta]\otimes\xi+
\eta(V)(\alpha \Id-\beta \phi)\\
&=\eta(V)\cdot\nabla \xi+d (\eta(V))\otimes\xi
\end{split}\label{ee}\\
\begin{split}
\pounds_{V}g&=d (\eta(V))\otimes\eta+\eta\otimes d (\eta(V))+
2\alpha\eta(V)(g-\eta\otimes\eta)\\
&=\eta(V)\cdot\pounds_{\xi}g+
d (\eta(V))\otimes\eta+\eta\otimes d (\eta(V))
\end{split}\label{pp}\\
\div(V)&=2n\alpha\eta(V)+\xi(\eta(V))
=\eta(V)\cdot\div(\xi)+(d (\eta(V)))(\xi).\label{hh}
\end{align}
\end{lemma}

\begin{proof}
For any $X$, $Y\in\mathfrak{X}(M)$, we have
\begin{align*}
\nabla_X V={}&X(\eta(V))\xi+\eta(V)\nabla_X\xi\\
={}&[X(\eta(V))-\alpha\eta(V)\eta(X)]\xi+\eta(V)(\alpha X-\beta\phi X)\\
(\pounds_{V}g)(X,Y)={}&g(\nabla_XV,Y)+g(X,\nabla_YV)\\
={}&X(\eta(V))\eta(Y)+Y(\eta(V))\eta(X)\\
&-2\alpha\eta(V)\eta(X)\eta(Y)+2\alpha\eta(V)g(X,Y)\\
\div(V)={}&\sum_{i=1}^{2n+1}g(\nabla_{E_i}V,E_i)\\
={}&\sum_{i=1}^{2n+1}\{[E_i(\eta(V))-\alpha \eta(V)\eta(E_i)]\eta(E_i)\\
&+\eta(V)[\alpha g(E_i,E_i)-\beta g(\phi E_i,E_i)]\}\\
={}&2n\alpha\eta(V)+\xi(\eta(V)),
\end{align*}
where $\{E_i\}_{1\le i\le 2n+1}$ is a $g$-orthonormal frame field on $M$.
Now using equations~\eqref{16} we obtain the relations~\eqref{ee}--\eqref{hh}.
\end{proof}


\begin{example}
Consider the $3$-dimensional Kenmotsu manifold $(M, \phi,\xi,\eta,g)$, where $M=\{(x,y,z)\in\mathbb{R}^3| z>1\}$, with $(x,y,z)$ the standard coordinates in $\mathbb{R}^3$, and
\begin{align*}
\begin{split}
\phi&= d x\otimes \frac{\displaystyle \partial}{\displaystyle \partial y}-d y\otimes \frac{\displaystyle \partial}{\displaystyle \partial x},\qquad
\xi=\frac{\displaystyle \partial}{\displaystyle \partial z},\qquad
\eta= d z,\\
g&=e^{2z}(d x\otimes d x+d y\otimes d y)+d z\otimes d z.
\end{split}
\end{align*}
An orthonormal frame field on $TM$ is given by
$$E_1=e^{-z}\frac{\partial}{\partial x},\qquad E_2=e^{-z}\frac{\partial}{\partial y},\qquad E_3=\frac{\partial}{\partial z}\cdot$$
The nonzero components of the Levi-Civita connection are
$$\nabla_{E_1}E_1=-E_3=\nabla_{E_2}E_2,\qquad \nabla_{E_1}E_3=E_1,\qquad \nabla_{E_2}E_3=E_2,$$
the nonzero components of the Riemann and Ricci curvatures are
\begin{align*}
R(E_1,E_2)E_2=&-E_1=R(E_1,E_3)E_3\\
R(E_2,E_1)E_1=&-E_2=R(E_2,E_3)E_3\\
R(E_3,E_1)E_1=&-E_3=R(E_3,E_2)E_2\\
\Ric(E_i,E_i)=&-2,\qquad i\in \{1,2,3\}.
\end{align*}
For $V=e^z\frac{\displaystyle \partial}{\displaystyle \partial z}$, the Lie derivative of $g$ in the direction of $V$ is given by
$$(\pounds _{V}g)(E_1,E_1)=(\pounds _{V}g)(E_2,E_2)=(\pounds _{V}g)(E_3,E_3)=2e^z.$$
Then the pair \ $(V=e^z\frac{\displaystyle \partial}{\displaystyle \partial z}, \ \lambda=2e^z-1)$ \ defines an almost Riemann soliton, respectively, \ $(V=e^z\frac{\displaystyle \partial}{\displaystyle \partial z}, \ \lambda=e^z-2)$ \ defines an almost Ricci soliton.
\end{example}

\section{Almost Riemann and almost Ricci solitons}

Consider now $(V,\lambda)$ an almost Riemann soliton on $M$ and assume that
the potential vector field $V$ is pointwise collinear with $\xi$, i.e.,
$V=\eta(V)\xi$.

Expressing the Riemann soliton equation~\eqref{2}, we get
\begin{equation}\label{3}
\begin{gathered}
2R(X,Y,Z,W)+g(X,W)(\pounds_{V}g)(Y,Z)+g(Y,Z)(\pounds_{V}g)(X,W)\\
-g(X,Z)(\pounds_{V}g)(Y,W)-g(Y,W)(\pounds_{V}g)(X,Z)\\
=2\lambda[g(X,W)g(Y,Z)-g(X,Z)g(Y,W)]
\end{gathered}
\end{equation}
and contracting it over $X$ and $W$, we find
\begin{equation}\label{4}
\frac{1}{2}\pounds_{V}g+\frac{1}{2n-1}\Ric=\frac{2n\lambda-\div(V)}{2n-1}g
\end{equation}
and further
\begin{equation}\label{9}
\scal=2n[(2n+1)\lambda-2\div(V)].
\end{equation}

Then replacing equation~\eqref{pp} in equalities~\eqref{4} and~\eqref{9}, we obtain
\begin{equation}\label{22}
\begin{aligned}
\Ric={}&-\frac{2n-1}{2}[d (\eta(V))\otimes\eta+\eta\otimes d (\eta(V))-
2\alpha\eta(V)\eta\otimes\eta]\\
&+[2n\lambda-(4n-1)\alpha\eta(V)-\xi(\eta(V))]g\\
Q={}&-\frac{2n-1}{2}\{[d (\eta(V))-2\alpha\eta(V)\eta]\otimes\xi+
\eta\otimes\grad(\eta(V))\}\\
&+[2n\lambda-(4n-1)\alpha\eta(V)-\xi(\eta(V))]\Id
\end{aligned}
\end{equation}
and
\begin{equation}\label{s2}
\scal=2n[(2n+1)\lambda-4n\alpha\eta(V)-2\xi(\eta(V))],
\end{equation}
where $Q$ is the Ricci operator defined by $g(QX,Y):=\Ric(X,Y)$.

Consider now $(V,\lambda)$ an almost Ricci soliton on
$M$ and assume that the potential vector field $V$ is
pointwise collinear with $\xi$, i.e., $V=\eta(V)\xi$.

Contracting the Ricci soliton equation~\eqref{1} and considering~\eqref{pp},
we get
\begin{align}
\Ric={}&-\frac{1}{2}[d (\eta(V))\otimes\eta+\eta\otimes d (\eta(V))]+
[\lambda-\alpha\eta(V)]g+\alpha\eta(V)\eta\otimes\eta\notag\\
Q={}&-\frac{1}{2}\{[d (\eta(V))-2\alpha\eta(V)\eta]\otimes\xi+
\eta\otimes\grad(\eta(V))\}\label{23}\\
&+[\lambda-\alpha\eta(V)]\Id\notag
\end{align}
and
\begin{equation}\label{s1}
\scal=(2n+1)\lambda-2n\alpha\eta(V)-\xi(\eta(V)).
\end{equation}

\begin{remark}
If the vector field $V$ is a constant multiple of $\xi$, then any $\beta$-Sasakian manifold
admitting $(V,\lambda)$ as almost Riemann, respectively, almost Ricci soliton, is an Einstein manifold of scalar curvature
$\lambda=\dfrac{\scal}{2n(2n+1)}$, respectively, $\lambda=\nolinebreak \dfrac{\scal}{2n+1}\cdot$
\end{remark}

\begin{proposition}
Let $(V,\lambda)$ define an almost Riemann or an almost Ricci soliton on
the $(\alpha,\beta)$-contact metric manifold $(M,\phi,\xi,\eta,g)$
such that the potential vector field $V$ is pointwise collinear with
$\xi$, i.e., $V=\eta(V)\xi$. Then
\begin{enumerate}
\item
$\phi\circ Q=Q\circ\phi$ \ if and only if \
$\eta\otimes \phi(\grad(\eta(V)))=[d (\eta(V))\circ\phi]\otimes\xi$
\item
$\phi^2\circ Q=Q\circ\phi^2$ \ if and only if \
$\eta\otimes \grad(\eta(V))=d (\eta(V))\otimes\xi$.
\end{enumerate}
\end{proposition}

\begin{proof}
From~\eqref{22} and~\eqref{23}, by a direct computation we get the conclusions.
\end{proof}

From the above considerations, we obtain

\begin{proposition}
If $(\xi,\lambda)$ defines an almost Riemann, respectively, an almost Ricci soliton on
the $(2n+1)$-di\-men\-sional $(\alpha,\beta)$-contact metric manifold
$(M,\phi,\xi,\eta,g)$, then
\begin{enumerate}
\item
$M$ is an almost quasi-Einstein manifold with the defining functions\\
$2n\lambda-(4n-1)\alpha$ \ and \ $(2n-1)\alpha$, respectively, \ $\lambda-\alpha$ \ and \ $\alpha$
\item
$\scal=2n[(2n+1)\lambda-4n\alpha]$, \ respectively, \ $\scal=(2n+1)\lambda-2n\alpha$.
\end{enumerate}
In particular, any $\beta$-Sasakian manifold admitting an almost Riemann or an almost Ricci
soliton $(\xi,\lambda)$ is an Einstein manifold.
\end{proposition}

\begin{proof}
In the Riemann soliton case we have
\[
\Ric=[2n\lambda-(4n-1)\alpha]g+(2n-1)\alpha\eta\otimes \eta
\]
and in the Ricci soliton case we have
\[
\Ric=(\lambda-\alpha)g+\alpha\eta\otimes \eta
\]
hence the conclusions.
\end{proof}

Moreover, we deduce

\begin{corollary}
If $(\xi,\lambda)$ defines a Riemann or a Ricci soliton on a Sasakian manifold,
then its scalar curvature is constant.
\end{corollary}

If $(M,\phi,\xi,\eta,g)$ is an $\alpha$-Kenmotsu manifold, then $\xi$ is
a torse-forming vector field, and we can state

\begin{proposition}
If $(\xi,\lambda)$ defines an almost Riemann soliton on the $(2n+1)$-di\-men\-sional
$\alpha$-Kenmotsu manifold $(M,\phi,\xi,\eta,g)$ satisfying
$R(\xi,\cdot)\cdot\Ric=0$, then
\[
\lambda=\alpha
\text{\qquad and\qquad}
\scal=-2n(2n-1)\alpha.
\]
In particular, under these hypotheses, a Kenmotsu manifold is of negative
scalar curvature.
\end{proposition}

\begin{proof}
Since $\nabla\xi=\alpha(\Id-\eta\otimes\xi)$, from~\eqref{3} we get
$$
R(X,Y)Z=(\lambda-2\alpha)[g(Y,Z)X-g(X,Z)Y]+\alpha\eta(Z)[\eta(Y)X-\eta(X)Y]$$$$
+\alpha[g(Y,Z)\eta(X)-g(X,Z)\eta(Y)]\xi
$$
and, in particular
\begin{equation}\label{8}
R(\xi,Y)Z=(\lambda-\alpha)[g(Y,Z)\xi-\eta(Z)Y].
\end{equation}

The condition that must be satisfied by $\Ric$ is
\begin{equation}\label{88}
\Ric(R(\xi,X)Y,Z)+\Ric(Y,R(\xi,X)Z)=0,
\end{equation}
for any $X$, $Y$, $Z\in\mathfrak{X}(M)$.

Replacing the expression of $R(\xi,\cdot)\cdot$ from~\eqref{8} in~\eqref{88},
we get
\begin{equation}\label{h}
(\lambda-\alpha)[g(X,Y)\Ric(\xi,Z)-\eta(Y)\Ric(X,Z)+g(X,Z)\Ric(\xi,Y)-\eta(Z)\Ric(Y,X)]=0,
\end{equation}
for any $X$, $Y$, $Z\in\mathfrak{X}(M)$. But in an $\alpha$-Kenmotsu manifold,
we have
\[
\Ric(\xi,X)=-[2n\alpha^2+\xi(\alpha)]\eta(X)-(2n-1)X(\alpha)
\]
and from the soliton condition we get
$\Ric(\xi,X)=2n(\lambda-\alpha)\eta(X)$,
which implies
\[
\grad(\alpha)=-\frac{1}{2n-1}[2n(\lambda-\alpha+\alpha^2)+\xi(\alpha)]\xi.
\]

Replacing $\Ric(\xi,\cdot)$ in~\eqref{h}, we obtain
$$
(\lambda-\alpha)\{2n(\lambda-\alpha)[g(X,Y)\eta(Z)+g(X,Z)\eta(Y)]
-[\eta(Y)\Ric(X,Z)+\eta(Z)\Ric(X,Y)]\}=0,
$$
for any $X$, $Y$, $Z\in\mathfrak{X}(M)$.

For $Z:=\xi$, we have
\[
(\lambda-\alpha)[2n(\lambda-\alpha)g-\Ric]=0
\]
and using again the soliton condition, we get $\lambda=\alpha$,
hence the conclusion.
\end{proof}

Taking $V=\xi$ and differentiating covariantly~\eqref{22}, we get
\begin{equation}\label{o}
\begin{aligned}
(\nabla_X\Ric)(Y,Z)={}&X(\Ric(Y,Z))-\Ric(\nabla_XY,Z)-\Ric(Y,\nabla_XZ)\\
={}&[2nX(\lambda)-(4n-1)X(\alpha)]g(Y,Z)\\
&+(2n-1)X(\alpha)\eta(Y)\eta(Z)\\
&+(2n-1)\alpha[(\nabla_X\eta)Y\cdot\eta(Z)+(\nabla_X\eta)Z\cdot\eta(Y)]
\end{aligned}
\end{equation}
and differentiating covariantly~\eqref{23}, we get
\begin{align}
(\nabla_X\Ric)(Y,Z)={}&X(\Ric(Y,Z))-\Ric(\nabla_XY,Z)-\Ric(Y,\nabla_XZ)\notag\\
={}&[X(\lambda)-X(\alpha)]g(Y,Z)+X(\alpha)\eta(Y)\eta(Z)\label{oo}\\
&+\alpha[(\nabla_X\eta)Y\cdot\eta(Z)+(\nabla_X\eta)Z\cdot\eta(Y)],\notag
\end{align}
for any $X$, $Y$, $Z\in\mathfrak{X}(M)$. Then we can state

\begin{proposition}
If $(\xi,\lambda)$ defines an almost Riemann, respectively, an almost Ricci soliton
on the $(2n+1)$-dimen\-sio\-nal $(\alpha,\beta)$-contact
metric manifold $(M,\phi,\xi,\eta,g)$ satisfying $\nabla \Ric =0$,
then $d \lambda=d \alpha$, respectively, $\lambda$ is a constant.
\end{proposition}

\begin{proof}
Taking $Y=Z=\xi$ in~\eqref{o}, the condition $\nabla\Ric=0$ implies
$X(\lambda-\alpha)=0$, for any $X\in\mathfrak{X}(M)$.

Taking $Y=Z=\xi$ in~\eqref{oo}, the condition $\nabla\Ric=0$ implies
$X(\lambda)=0$, for any $X\in\mathfrak{X}(M)$.

Therefore we get the conclusions.
\end{proof}

\begin{proposition}
Let $(\xi,\lambda)$ define an almost Riemann, respectively, an almost Ricci soliton
on the $(2n+1)$-di\-men\-sional $(\alpha,\beta)$-contact
metric manifold $(M,\phi,\xi,\eta,g)$.
\begin{enumerate}
\item
If $\nabla Q=0$, then $\alpha=0$ and $M$ is of constant scalar curvature
$2n(2n+1)\lambda$, respectively, $(2n+1)\lambda$.
\item
If $\phi^2\circ\nabla Q=0$, then $M$ is a cosymplectic manifold or
$\alpha=0$, $\beta\neq 0$, $\lambda$ is a constant and $M$ is of
constant scalar curvature $2n(2n+1)\lambda$, respectively, $(2n+1)\lambda$.
\end{enumerate}
\end{proposition}

\begin{proof}
In the Riemann soliton case, we have
\[
\begin{aligned}
(\nabla_XQ)Y={}&\nabla_XQY-Q(\nabla_XY)\\
={}&[2nX(\lambda)-(4n-1)X(\alpha)]Y+(2n-1)\alpha^2\eta(Y)X\\
&-(2n-1)\alpha\beta\eta(Y)\phi X\\
&+(2n-1)[X(\alpha)\eta(Y)-\alpha^2\eta(X)\eta(Y)+\alpha(\nabla_X\eta)Y]\xi,
\end{aligned}
\]
for any $X$, $Y\in\mathfrak{X}(M)$ and taking $Y=\xi$, from $\nabla Q=0$ we get
\[
[2nX(\lambda-\alpha)-(2n-1)\alpha^2\eta(X)]\xi=
(2n-1)\alpha(\beta\phi X-\alpha X).
\]
Applying $\eta$, we obtain $X(\lambda-\alpha)=0$, for any
$X\in\mathfrak{X}(M)$ and replacing it in the previous relation, we get
$\alpha\nabla_X\phi=0$, for any $X\in\mathfrak{X}(M)$, hence the conclusion i).

Now, from $\phi^2((\nabla_X Q)Y)=0$, we get
$$
[2nX(\lambda)-(4n-1)X(\alpha)][\alpha Y-\alpha \eta(Y)\xi-\beta\phi Y]
+(2n-1)\alpha^2\eta(Y)[\alpha X-\alpha\eta(X)\xi-\beta\phi X]$$$$+
(2n-1)\alpha\beta\eta(Y)\phi X=0
$$
and by taking $Y=\xi$, we obtain
\[
\alpha(\alpha\nabla_X\xi+\beta\phi X)=0,
\]
for any $X\in\mathfrak{X}(M)$, which implies $\alpha=0$.
Replacing it in the above relation, we get $\beta X(\lambda)=0$,
and we deduce ii).

In the Ricci soliton case, we have
\begin{align*}
(\nabla_XQ)Y={}&\nabla_XQY-Q(\nabla_XY)\\
={}&X(\lambda-\alpha)Y+\alpha^2\eta(Y)X-\alpha\beta\eta(Y)\phi X\\
&+[X(\alpha)\eta(Y)-\alpha^2\eta(X)\eta(Y)+\alpha(\nabla_X\eta)Y]\xi,
\end{align*}
for any $X$, $Y\in\mathfrak{X}(M)$ and taking $Y=\xi$, from $\nabla Q=0$
we get
\[
[X(\lambda)-\alpha^2\eta(X)]\xi=\alpha(\beta\phi X-\alpha X).
\]
Applying $\eta$, we obtain $X(\lambda)=0$, for any $X\in\mathfrak{X}(M)$ and
replacing it in the previous relation, we get $\alpha\nabla_X\phi=0$, for any
$X\in\mathfrak{X}(M)$, hence the conclusion i).

Now, from $\phi^2((\nabla_X Q)Y)=0$, we get
$$
X(\lambda-\alpha)[\alpha Y-\alpha\eta(Y)\xi-\beta\phi Y]+
\alpha^2\eta(Y)[\alpha X-\alpha\eta(X)\xi-\beta \phi X]
+\alpha\beta\eta(Y)\phi X=0
$$
and by taking $Y=\xi$, we obtain
\[
\alpha(\alpha\nabla_X\xi+\beta\phi X)=0,
\]
for any $X\in\mathfrak{X}(M)$, which implies $\alpha=0$.
Replacing it in the above relation, we get $\beta X(\lambda)=0$,
and in this way we deduce ii).
\end{proof}

\let\ol=\overline

\begin{remark}
If $(V,\lambda)$ defines an almost Riemann soliton and $(V,\ol{\lambda})$ defines an almost
Ricci soliton on the $(2n+1)$-dimensional
$(\alpha,\beta)$-contact metric manifold $(M,\phi,\xi,\eta,g)$, $n>1$, then $M$ is an Einstein manifold with
\[\Ric=\frac{2n}{4n-1}(2\ol{\lambda}-\lambda)g
\] and
\[
\ol{\lambda}=\xi(\eta(V))+2n[\beta^2-\alpha^2-\xi(\alpha)].
\]
Moreover, if $V$ is pointwise collinear with $\xi$,
then
\[
\lambda=2\xi(\eta(V))+\beta^2-\alpha^2-\xi(\alpha).
\]
Indeed, using~\eqref{1} and~\eqref{4} we get the expression of $\Ric$.
Now computing~\eqref{1} for $X=Y=\xi$ and taking into account that
$\Ric(\xi,\xi)=2n[\beta^2-\alpha^2-\xi(\alpha)]$~\cite{des}, we obtain $\ol{\lambda}$ and
then also $\lambda$.
\end{remark}

\begin{remark}
If $(V,\lambda)$ defines an almost Riemann soliton on the $(2n+1)$-dimensional
$(\alpha,\beta)$-contact metric manifold $(M,\phi,\xi,\eta,g)$, $n>1$, then $(\ol{V}, \ol{\lambda})$, for $\ol{V}=(2n-1)V$, $\ol{\lambda}=2n\lambda-\div(V)$, defines an almost
Ricci soliton. Moreover, if $V$ is pointwise collinear with $\xi$,
then
\begin{align*}
\lambda={}&\xi(\eta(V))+\alpha\eta(V)+\beta^2-\alpha^2-\xi(\alpha)\\
\ol{\lambda}={}&(2n-1)\xi(\eta(V))+2n[\beta^2-\alpha^2-\xi(\alpha)].
\end{align*}
Indeed, from~\eqref{1} and~\eqref{4} we deduce that $(\ol{V}=(2n-1)V, \ol{\lambda}=2n\lambda-\div(V))$ defines an almost
Ricci soliton. Computing~\eqref{4} for $X=Y=\xi$ and using $\Ric(\xi,\xi)=2n[\beta^2-\alpha^2-\xi(\alpha)]$
and~\eqref{hh} we obtain the expressions of $\lambda$ and $\ol{\lambda}$.
\end{remark}

\begin{remark}
We saw that any almost Riemann soliton $(V,\lambda)$ on an $m$-dimensional Riemannian manifold $(M,g)$, gives rise to an almost Ricci soliton $(\ol{V}, \ol{\lambda})$, for $\ol{V}=(m-2)V$, $\ol{\lambda}=(m-1)\lambda-\div(V)$. Moreover, if $\grad(\lambda)=\frac{1}{m-1}\grad(\div(V))$, then an almost Riemann soliton determines a Ricci soliton.

For $m\geq 3$
we have the decomposition of the Riemann curvature tensor with respect to the Weil curvature tensor $\mathcal{W}$
$$R=\mathcal{W}-\frac{\scal}{2(m-1)(m-2)}g\odot g+\frac{1}{m-2}\Ric\odot g.$$

If $(V,\lambda)$ defines an almost Riemann soliton on $M$, then
\begin{align*}
R={}&-\frac{1}{2}\pounds_{V}g\odot g+\frac{1}{2}\lambda g\odot g\\
={}&\frac{1}{m-2}\Ric\odot g+\frac{1}{2(m-2)}[2\div(V)-m\lambda]g\odot g\\
={}&\frac{1}{m-2}\Ric\odot g-\frac{\scal}{2(m-1)(m-2)}g\odot g
\end{align*}
which implies that
the Weil curvature tensor must vanish.

On the other hand, if $(V,\lambda)$ defines an almost Ricci soliton on $M$, then
\[
\frac{1}{2}\pounds_{V}g\odot g=-\Ric\odot g+\lambda g\odot g
\]
therefore it will determine an almost Riemann soliton with the potential vector field $V$ if and only if there exists a smooth function $\ol{\lambda}$ such that
\[
R=(\Ric +\frac{1}{2}(\ol{\lambda}-2\lambda)g)\odot g.
\]
\end{remark}

Now we show that the only $(\alpha,\beta)$-contact metric manifolds for which
$(\xi,\lambda)$ defines two types of the following solitons: almost Riemann,
almost Ricci, almost Yamabe, are Ricci-flat cosymplectic manifolds. Recall
that $(V,\lambda)$ defines an almost Yamabe soliton on $(M,g)$ if
\begin{equation}\label{y}
\pounds_{V}g=(\lambda-\scal)g.
\end{equation}

\begin{proposition}
Let $(M,\phi,\xi,\eta,g)$ be a $(2n+1)$-di\-men\-sional $(\alpha,\beta)$-contact metric manifold.
If $(\xi,\lambda)$ defines an almost Riemann and an almost Ricci soliton,
or an almost Riemann and an almost Yamabe soliton, or an almost Ricci and
an almost Yamabe soliton on $M$, then
\[
\alpha=0,
\qquad
\beta=0,
\qquad
\lambda=0,
\qquad
\Ric=0
\]
hence $M$ is a Ricci-flat cosymplectic manifold.
\end{proposition}

\begin{proof}
Equating $\pounds_{\xi}g$ from the soliton's equations~\eqref{4},~\eqref{1}
and~\eqref{y}, we obtain $\alpha{=}0$, \ $\lambda=0$ \ and \ $\Ric=0$. From \ $\Ric(\xi,\xi)=2n[\beta^2-\alpha^2-\xi(\alpha)]$
\ we get also $\beta=0$.
\end{proof}

Adara M. Blaga\\
West University of Timi\c{s}oara\\
Timi\c{s}oara, 300223, ROM\^{A}NIA\\
\emph{E-mail address}: {\tt adarablaga@yahoo.com}

\smallskip
Dan Radu La\c tcu\\
West University of Timi\c{s}oara\\
Timi\c{s}oara, 300223, ROM\^{A}NIA\\
\emph{E-mail address}: {\tt latcu07@yahoo.com}

\end{document}